\documentclass{article}

\usepackage{amsmath, amssymb}
\usepackage{amsthm} 
\usepackage{enumerate}  
\usepackage{url} 
\usepackage{makeidx} 
\usepackage{hyperref}
\usepackage{tikz}

\usepackage{graphicx}

\newtheoremstyle{plainsl}%
	{\topsep}
	{\topsep}
	{\slshape} 
	{}
	{\normalfont\bfseries}
	{.}
	{ }
	{}

\swapnumbers

{\theoremstyle{plainsl}
\newtheorem{theorem}{Theorem}[section]
\newtheorem{lemma}[theorem]{Lemma}
\newtheorem{corollary}[theorem]{Corollary}
\newtheorem{proposition}[theorem]{Proposition}

\newtheorem{problem}[theorem]{Problem}}

{\theoremstyle{remark}
}

\newcommand\cref[1]{Corollary~\ref{cor:#1}}

\newcommand\sqr[2]{{\vbox{\hrule height.#2pt
    \hbox{\vrule width.#2pt height#1pt \kern#1pt
        \vrule width.#2pt}\hrule height.#2pt}}}
\renewcommand\qed{%
	\ifmmode\eqno\sqr53
	\else\nolinebreak\ \hfill\sqr53\medbreak\fi}

\numberwithin{equation}{section}

\DeclareMathOperator\rk{rk}

\DeclareMathOperator\row{row}
\DeclareMathOperator\ns{null}

\DeclareMathOperator\GF{GF}

\newcommand\ints{{\mathbb Z}}

\title{Hardness of Computing Clique Number and Chromatic Number For Cayley Graphs}
\author{Chris Godsil and Brendan Rooney}
\date{\today}

\begin{document}

\maketitle

\begin{abstract}
Computing the clique number and chromatic number of a general graph are well-known NP-Hard problems. Codenotti et al. (Bruno Codenotti, Ivan Gerace, and Sebastiano Vigna. Hardness results and spectral techniques for combinatorial problems on circulant graphs. \emph{Linear Algebra Appl.}, 285(1-3): 123--142, 1998) showed that computing clique number and chromatic number are still NP-Hard problems for the class of circulant graphs. We show that computing clique number is NP-Hard for the class of Cayley graphs for the groups $G^n$, where $G$ is any fixed finite group (e.g., cubelike graphs). We also show that computing chromatic number cannot be done in polynomial time (under the assumption $\text{P}\neq \text{NP}$) for the same class of graphs. Our presentation uses free Cayley graphs. The proof combines free Cayley graphs with quotient graphs and Goppa codes.
\end{abstract}

\let\thefootnote\relax\footnote{\emph{AMS Classification:} 05C15, 05C50, 05C85, 68R10}
\let\thefootnote\relax\footnote{\emph{Keywords:} Chromatic Number, Clique Number, Computational Complexity, Cayley Graphs, Cubelike Graphs, Goppa Codes}


In his celebrated 1972 paper \cite{Karp}, Karp established the NP-Completeness of 21 combinatorial problems. Amongst those problems are the CLIQUE problem and the CHROMATIC NUMBER problem. CLIQUE takes a graph $X$ and an integer $k$ and decides whether $X$ contains a clique of size $k$ as a subgraph. CHROMATIC NUMBER takes a graph $X$ and an integer $k$ and decides whether there is a proper colouring of $X$ using at most $k$ colours.

The \emph{clique number} of a graph $X$ is the size of the largest clique contained in $X$, and is denoted by $\omega(X)$. Since deciding whether a general graph $X$ contains a clique of size $k$ is NP-Complete, the problem of computing $\omega(X)$ is NP-Hard. The \emph{chromatic number} of a graph $X$ is the smallest integer $k$ such that $X$ has a proper $k$-colouring, and is denoted by $\chi(X)$. Again, since deciding whether a general graph $X$ can be coloured properly using at most $k$ colours is NP-Complete, computing $\chi(X)$ is NP-Hard.

Some of the graph theoretic problems in Karp's list become easier when restricted to a subclass of graphs. For instance, deciding whether a graph $X$ has a subset of vertices with size $k$ that covers all of the edges of $X$ is NP-Complete. However, if $X$ is bipartite the Hungarian Algorithm finds a minimum vertex cover of $X$ in polynomial time. There are also subclasses of graphs for which computing clique number and chromatic number are computable in polynomial time. For example, planar graphs have polynomial time computable clique numbers, and graphs with treewidth at most $k$ have polynomial time computable chromatic numbers \cite{Arnborg}.

In 1998, Codenotti, Gerace, and Vigna \cite{Codenotti} proved that computing clique number, and chromatic number, are NP-Hard when restricted to the class of circulant graphs (a \emph{circulant} is a Cayley graph for a group $\ints_m$). Since circulants are Cayley graphs, they are vertex transitive. One might hope, or expect, that the assumption of vertex transitivity would confer some advantage when approaching computational problems on graphs. Codenotti et al.'s results show that this is not the case, and also  raise the question of whether there are classes of Cayley graphs on which these problems become easier.

Our main results in this paper are analogues of Codenotti et al.'s hardness results for a different class of Cayley graphs. Our results apply to the class of Cayley graphs for the groups $G^n$, where $G$ is any fixed finite group. When $G=\ints_2$, this is the class of \emph{cubelike graphs}. We show that computing clique number for these graphs is an NP-Hard problem (Theorem \ref{MainTheorem}). We also show that computing chromatic number for these graphs cannot be done in polynomial time under the assumption that $\text{P}\neq\text{NP}$ (Theorem \ref{ChromCor}).

We prove that computing clique number is NP-Hard by reducing computing the clique number of a general graph $X$ to computing the clique number of a Cayley graph on a group $G^n$. The key to this reduction is providing a construction of a Cayley graph $\Gamma$ from $X$ so that $|\Gamma|$ is polynomially bounded in $|X|$, and so that $\omega(X)$ is easily computed from $\omega(\Gamma)$. We begin with a construction used by Babai and S{\'o}s \cite{BabaiSos} to embed graphs in Cayley graphs. This construction leads naturally to free Cayley graphs. We construct a free Cayley graph $G(X)$ from $X$ so that the cliques in $X$ can be recovered from the cliques in $G(X)$. To complete our construction we will quotient $G(X)$ over a suitably chosen linear code. Specifically, we will give a Goppa code that satisfies the desired properties.

To prove that computing chromatic number cannot be done in polynomial time we use a similar strategy to Codenotti et al. In \cite{Codenotti}, the chromatic number result is proven by reducing computing clique number for general graphs to computing chromatic number for circulant graphs. This reduction does not work for the Cayley graphs we consider. However, we adapt this approach to show that if chromatic number can be computed in polynomial time for Cayley graphs for the groups $G^n$, then for any graph $X$, the clique number of $X$ can be approximated to within a constant factor in polynomial time. This completes the proof by an inapproximability result of H{\aa}stad \cite{Hastad}.

Constructing our reductions using free Cayley graphs situates them in a more general framework. Free Cayley graphs are relatively new and unstudied objects (they appear in a recent paper by Beaudou, Naserasr and Tardif \cite{Beaudou}). Our complexity results rely on the clique structure of free Cayley graphs.

This paper has roughly three parts. The first part gives the background necessary for our results and their proofs. Sections \ref{QuotientGraphs} and \ref{GoppaCodes} contain some basic graph theory and coding theory. Sections \ref{Embeddings}, \ref{FreeCayleyGraphs} and \ref{AuxiliaryGraphs} introduce free Cayley graphs. The second part contains the intermediate results needed for our reductions, and the proofs of our main results. Section \ref{CliquePrelim} develops the relationship between cliques in a graph $X$ and cliques in the free Cayley graph $G(X)$. Sections \ref{Zp} and \ref{LinearCodes} focus on the groups $\ints_p$ for $p$ prime. The main results of this paper are given in Sections \ref{CliqueNumber} and \ref{ChromaticNumber}. The last part consists of some additional observations. We consider how our construction can be applied to embeddings in Section \ref{EmbeddingsRe}, and we take a closer look at free Cayley graphs in Section \ref{Structure}.

\section{Graphs, and Quotient Graphs}\label{QuotientGraphs}

The main objects in this paper are Cayley graphs. Instead of giving an exhaustive list of definitions, we refer the reader to Godsil and Royle \cite{GodsilRoyle} for the basics. In this section, we give a selection of definitions and notation. We also spend some time reviewing quotients of graphs, and of Cayley graphs, as quotients are a crucial tool in our construction.

In this paper we will typically refer to graphs using $X$ and $Y$, and groups using $G$ and $H$. For a vertex $i$ of a graph $X$, the \emph{neighbourhood of $i$ in $X$} is the set of all vertices $j$ that are adjacent to $i$. If $S\subseteq V(X)$, then we denote the subgraph of $X$ induced by $S$ by $X[S]$. We denote the subgraph of $X$ induced by the neighbourhood of $i$ by $X[i]$.

The \emph{independence number} of $X$ is the size of a largest independent set (or coclique), denoted $\alpha(X)$. We denote the \emph{line graph} of a graph $X$ by $L(X)$. We denote the \emph{complete graph} on $v$ vertices by $K_v$. From graphs $X$ and $Y$, we define the \emph{Cartesian product} $X\Box Y$ as follows. The vertex set of $X\Box Y$ is $V(X)\times V(Y)$. Two vertices $(a,x)$ and $(b,y)$ are adjacent if and only if either $a=b$ and $x$ is adjacent to $y$ in $Y$, or $a$ is adjacent to $b$ in $X$ and $x=y$.

Let $X$ and $Y$ be graphs, and $h:X\rightarrow Y$ be a homomorphism (a map that preserves edges). If for all $x\in V(X)$, the map induced by $h$ from the neighbours of $x$ to the neighbours of $h(x)$ is a bijection, then $h$ is a \emph{local isomorphism}. The map $h$ is a {\it covering map} if $h$ is a local isomorphism and a surjection. We say that $X$ is a {\it cover} of $Y$. If $|h^{-1}(y)|=r$ for all $y\in V(Y)$, we say that $X$ is an {\it $r$-fold cover} of $Y$.

If $G$ is a finite group, and $C\subseteq G$, then the \emph{Cayley graph} $X=X(G,C)$ has vertex set $V(X)=G$, and $a,b\in G$ are adjacent in $X$ if and only if $ab^{-1}\in C$. The set $C$ is the \emph{connection set}. Note that in order for this construction to produce a graph, we must have that $a^{-1}\in C$ for all $a\in C$. Also note that $X$ will have no loops if and only if $\text{Id}\notin C$, and if $\text{Id}\in C$ then $X$ has a loop on each vertex. Cayley graphs are \emph{vertex transitive}, for any $a,b\in G$, there is an automorphism of $X$ mapping $a$ to $b$.

Let $X$ be a graph, and $\Pi$ be a partition of $V(X)$. The {\it quotient graph} $X/\Pi$ is the graph on the cells of $\Pi$ with adjacency defined as follows. For $A,B$ cells of $\Pi$, we add an edge between $A$ and $B$ for every pair of  vertices $a\in A$ and $b\in B$ so that $a$ and $b$ are adjacent in $X$. Note that this graph may have loops and multiple edges.

Subgroups give natural partitions of Cayley graphs. For $H\leq G$, the cosets of $H$ partition the elements of $G$ into cells of equal size. We denote the partition induced by $H$ as $\Pi_H$. If $X$ is a Cayley graph for the group $G$, then we denote the quotient graph of $X$ with respect to the partition $\Pi_H$ as $X_H=X/\Pi_H$. Using a partition of $G$ into cosets of $H$, rather than an arbitrary partition, gives $\Pi_H$ additional structure.

If $\Pi=\{\pi_1,\ldots,\pi_k\}$ is a partition of the vertices of a graph $X$ we call $\Pi$ an {\it equitable partition} if for every $x\in\pi_i$, the number of neighbours of $x$ in $\pi_j$ depends only on $i$ and $j$. The following result is folklore and well-known. The proof is straightforward, and provided as Proposition 3.6.1 in \cite{Rooney}.
\begin{proposition}\label{CosetPartitionEquitable}
If $H\leq G$, then $\Pi_H$ is an equitable partition of any Cayley graph $X$ for $G$.\qed
\end{proposition}
\noindent As an immediate consequence, we have that if $X=X(G,C)$, and $H\leq G$ induces partition $\Pi_H$, then $X_H$ is a regular multigraph (i.e., each pair of, not necessarily distinct, vertices is connected by the same number of edges).

For quotient graphs of this form, we will abusively use $X_H$ to denote the simplification of this graph. That is, we take $X_H$ to be the graph $X/{\Pi_H}$ with loops deleted and multiple edges replaced by single edges.

For Abelian groups (e.g., $G=\ints_p^n$), we can say more about the simplified graph $X_H$. Given a group $G$ and a normal subgroup $H$ of $G$, the {\it quotient group} $G/H$ is the group on the cosets of $H$ in $G$ with group operation defined as follows. Given cosets $H+a$ and $H+b$ we define
\[
(H+a)+(H+b)=H+(a+b).
\]
It is straightforward to show that this operation is well-defined, and defines a group.

If $H$ is a normal subgroup of $G$, then the graph $X_H$ is a Cayley graph for the quotient group. This fact appears as Lemma 2.4 in \cite{Watkins}, and we re-state it here.
\begin{proposition}\label{QuotientCayley}
If $X=X(G,C)$ and $H$ is a normal subgroup of $G$, then $X_H=X(G/H,C')$ where $C'=\{H+g\,:\,g\in C\setminus H\}$.\qed
\end{proposition}
\noindent When $G$ is Abelian, all subgroups of $G$ are normal. So Proposition \ref{QuotientCayley} implies that for any $H\leq G$, the graph $X_H$ is a Cayley graph for $G/H$.

\section{Linear Codes, and Goppa Codes}\label{GoppaCodes}

We will use linear codes to construct quotient graphs for Cayley graphs on groups $\ints_p^n$ for prime $p$. We give a brief account of the properties of a code that will allow us to construct these quotient graphs while maintaining certain properties of the original graph. We refer the reader to MacWilliams and Sloane \cite{MacWilliamsSloane} for a proper introduction.

Given a prime power $q$ and an integer $n$, a {\it $q$-ary linear code} (or {\it code}) is a subspace $D$ of the vector space $\GF(q)^n$. The {\it block length} of $D$ is the length of the vectors in $D$ (or the dimension of the ambient space, $\GF(q)^n$). The {\it size} of $D$ is the number of vectors in $D$, and is equal to $q^k$ where $k$ is the dimension of $D$ as a subspace of $\GF(q)^n$.

The {\it distance} (or {\it Hamming distance}) between two vectors in $x,y\in \GF(q)^n$ is the number of indices $1\leq i\leq n$ for which $x_i\neq y_i$. We denote the distance between $x$ and $y$ by $d(x,y)$. Given a code $D$, the {\it minimum distance} (or {\it distance}) of $D$ is the minimum distance between any two elements of $D$,
\[
d=\min\{d(x,y)\,:\,x,y\in D\}.
\]

The {\it weight}, $w(x)$, of a vector $x\in D$ is the number of indices $1\leq i\leq n$ so that $x_i\neq 0$. So ${\bf 0}\in D$ is the unique codeword with weight zero. Note that if $z\in \GF(q)^n$,
\[
d(x,y)=d(x-z,y-z).
\]
So
\[
d(x,y)=d(x-y,{\bf 0})=w(x-y).
\]
Since $D$ is a subspace, $x-y\in D$ for all $x,y\in D$, and we can express the distance of $D$ as the minimum weight of a non-zero codeword,
\[
d=\min\{w(x)\,:\,x\in D\setminus\{{\bf 0}\}\}.
\]

A linear code $D$ is a subspace of a vector space $\GF(q)^n$; so $D$ has a basis, and can be expressed as the row space of a matrix $B$. We call $B$ the {\it generator matrix} of $D$. If $D$ has rank $k$, and block length $n$, then $B$ is a $k\times n$ matrix with elements from $\GF(q)$. We can convert $B$ into reduced row-echelon form, and so we may assume that $B$ takes the form
\[
B=[I_k | A]
\]
where $I_k$ is the $k\times k$ identity matrix, and $A$ is a $k\times (n-k)$ matrix over $\GF(q)$ (this is the \emph{standard form} of a generator matrix). The {\it dual code} of $D$ is the code defined by the generator matrix
\[
H=[-A^T|I_{n-k}].
\]
Since $H$ is a $(n-k)\times n$ matrix over $\GF(q)$, the code generated by $H$ has rank $n-k$ and block length $n$. Note that
\[
BH^T=HB^T=0,
\]
so $D=\ker(H^T)$. The matrix $H$ is called the {\it parity check matrix} for the code $D$.

Note that given a generator matrix $B$ for a code $D$, we can easily (i.e., in time polynomial in the length of $D$) find a parity check matrix for $D$. Likewise, given a parity check matrix for $D$ we can easily find a generator matrix for $D$.

We will make use of a specific class of linear codes, Goppa codes. In the remainder this section we give a brief description of Goppa codes, their properties and construction. Again we refer the reader to MacWilliams and Sloane for a more complete treatment (see Chapter 12, Section 3 of \cite{MacWilliamsSloane}).

A Goppa code is a linear code over a finite field $\GF(q)$ (where $q$ is any prime power). In order to specify the code we need two ingredients: a polynomial $g(x)$ whose coefficients are elements of $\GF(q^m)$; and a set $L\subseteq \GF(q^m)$ of non-roots of $g(x)$. The polynomial $g(x)$ is called the {\it Goppa polynomial}.

Let $L=\{\alpha_1,\ldots,\alpha_n\}$ be a subset of the non-roots of $g(x)$. Given a vector $a\in \GF(q)^n$, we define the rational function
\[
R_a(x)=\sum_{i=1}^n\frac{a_i}{(x-\alpha_i)}.
\]
The {\it Goppa code} $D(g,L)$ is the set of all vectors $a\in \GF(q)^n$ such that $R_a(x)=0$ in the polynomial ring $\GF(q^m)[x]/g(x)$. Note that $D(g,L)$ is a $q$-ary linear code with block length $n$.

To construct a Goppa code, we give a construction for a parity check matrix $H$ for the code, and use $H$ to find a generator matrix. Let $H'$ be the matrix whose entries are defined as
\[
H'[i,j]=\alpha_j^ig(\alpha_j)^{-1}
\]
for $1\leq i\leq r$ and $1\leq j\leq n$. The matrix $H'$ is a $r \times n$ matrix with entries in $\GF(q^m)$. Then $H$ is the matrix whose entries are obtained by replacing $H'[i,j]$ with the column vector in $\GF(q)^m$ corresponding to $H'[i,j]\in \GF(q^m)$ (i.e., using the standard representation of $\GF(q^m)$ as a set of polynomials of degree at most $m-1$ in $\GF(q)[x]$). Now $H$ is a $rm\times n$ matrix with entries in $\GF(q)$.

In this case $H$ may not have full rank. However,  $H$ is a parity check matrix for $D(g,L)$, and we can construct a matrix $B$ from $H$ with $\row(B)=\ns(H)$. Since $\rk(H)+\ns(H)=n$ and $\rk(H)\leq rm$, it follows that the rank of $D(g,L)$ is $k\geq n-rm$.

Finally, we need some information on the distance of Goppa codes. In general, if the Goppa polynomial $g(x)$ has rank $r$, then the code $D(d,L)$ will have distance $d\geq r+1$ \cite{MacWilliamsSloane}. Note that for both the rank and distance of $D(g,L)$ the specific values will depend on the polynomial chosen to construct the code. However, the bounds $k\geq n-rm$ and $d\geq r+1$ will suffice for our application.

\section{Embedding Graphs in Cayley Graphs}\label{Embeddings}

Let $X$ be a graph with vertex set $\{1,\ldots,v\}$, and $G$ be a finite group. Consider the following construction. Let $h:V(X)\rightarrow G$ be an assignment of vertices to group elements, and denote $h(i)=g_i$ (where the $g_i$ are not necessarily distinct). Now take $Y$ to be the Cayley graph $Y=X(G,\mathcal{C})$ where
\[
\mathcal{C} = \{g_ig_j^{-1}\,:\,\text{$i,j$ adjacent in $X$}\}.
\]
Depending on the assignment $h$, $Y$ and $X$ may share some structure.

The Cayley graph $Y=X(G,\mathcal{C})$ is connected if and only if there is a path from $0_G$ to each $g\in G$. So $Y$ is connected if and only if each $g\in G$ can be expressed as a product of elements of $\mathcal{C}$, or if and only if $\mathcal{C}$ is a generating set for $G$. Since in the above definition we don't require that $\mathcal{C}$ generates $G$, we are not guaranteed that $Y$ is connected. In fact, for the applications in this paper, connectedness is irrelevant, and our auxiliary graphs frequently will not be connected.

Note that $h:V(X)\rightarrow V(Y)$ is a homomorphism by construction (if $\{i,j\}\in E(X)$, then $g_ig_j^{-1}$ and $g_jg_i^{-1}$ are elements of $\mathcal{C}$, so $\{g_i,g_j\}\in E(Y)$). This does not immediately imply that we can derive structural properties of $X$ from those of $Y$. For example, if $h(i)=g$ for all $1\leq i\leq v$, then $Y$ is the disjoint union of $|G|$ loops. However, there are two straightforward conditions that ensure $h$ gives an embedding of $X$ in $Y$. First, we must have that each of the $g_i$ are distinct. If $h$ is an injection, then by our previous remarks, $Y$ will contain $X$ as a subgraph on the vertices $\{g_1,\ldots,g_v\}$. To ensure that $h$ is an embedding, we need $\{g_1,\ldots,g_v\}$ to induce a copy of $X$ in $Y$. This is achieved by requiring that $g_ig_j^{-1}\notin\mathcal{C}$ for all $\{i,j\}\notin E(X)$. To find an assignment $h$ satisfying the second condition, we need to know exactly the edges and non-edges of $X$, together with any relations satisfied by the elements of $G$. Alternatively, we can strengthen this condition by requiring that $g_ig_j^{-1}\neq g_kg_l^{-1}$ whenever $|\{i,j,k,l\}|\geq 3$.

Sets $\{g_1,\ldots,g_v\}\subseteq G$ that satisfy $g_ig_j^{-1}\neq g_kg_l^{-1}$ whenever $|\{i,j,k,l\}|\geq 3$ are referred to as \emph{Sidon sets of the second kind} by Babai and S{\'o}s in \cite{BabaiSos}. Note that this condition immediately implies that $h$ is an injection, and that $h$ gives an embedding of $X$ in $Y$. In fact, it is easy to see that if $\{g_1,\ldots,g_v\}\subseteq G$ is a Sidon set of the second kind, then any graph on $v$ vertices can be embedded in a Cayley graph for $G$ (this is given as Proposition 3.1 in \cite{BabaiSos}). 

In the following sections we will be mainly focussed on finite Abelian groups $G$. Using additive notation, $\{g_1,\ldots,g_v\}\subseteq G$ is a Sidon set of the second kind if $g_i-g_j\neq g_k-g_l$ whenever $|\{i,j,k,l\}|\geq 3$. Since $G$ is Abelian, this condition is equivalent to the condition that $g_i+g_j\neq g_k+g_l$ whenever $|\{i,j,k,l\}|\geq 3$, or that the 2-sums $g_i+g_j$ of elements of $\{g_1,\ldots,g_v\}$ are all distinct.

\section{Free Cayley Graphs}\label{FreeCayleyGraphs}

Following Neumann \cite{Neumann}, a class of groups $\mathcal{V}$ is a \emph{variety} if and only if it is closed with respect to taking subcartesian products and epimorphic images. The \emph{variety generated by $G$} consists of all homomorphic images of subgroups of direct products of $G$. We will refer to this variety as $\mathcal{V}_G$.

A group $G$ is \emph{relatively free} if it contains a generating set $S$ such that every mapping of $S$ into $G$ extends to an endomorphism. The set $S$ is called a \emph{set of free generators}. The variety $\mathcal{V}_G$ contains a relatively free group with $k$ free generators for every positive integer $k$. More precisely, for every positive integer $k$, there is a unique group $\mathcal{F}_k(G)$ such that any group in $\mathcal{V}_G$ which is generated by a set of $k$ elements is a homomorphic image of $\mathcal{F}_k(G)$. We will call $\mathcal{F}_k(G)$ the \emph{relatively free group on $k$ generators in $\mathcal{V}_G$}. 

For a finite group $G$ and integer $k$, let $g_1,\ldots,g_k$ be elements of $G^{|G|^k}$ so that each ordered $k$-tuple of elements of $G$ appears as some $(g_1[i],\ldots,g_k[i])$. Then we can take $\mathcal{F}_k(G)=\langle g_1,\ldots,g_k\rangle$. The generators $g_1,\ldots,g_k$ are called the {\it canonical generators} for $\mathcal{F}_k(G)$.

Let $G$ be a finite group, and let $X$ be a graph on $v$ vertices. Denote the canonical generators of the relatively free group $\mathcal{F}_v$ by $g_1,\ldots,g_v$. The {\it free Cayley graph} is the Cayley graph $G(X)=X(\mathcal{F}_v,\mathcal{C})$ where $\mathcal{C}$ is the connection set
\[
\mathcal{C} = \{g_ig_j^{-1}\,:\,\text{$i,j$ adjacent in $X$}\}.
\]
As in the previous section, the function $h:V(X)\rightarrow G(X)$ given by $h(i)=g_i$ is a homomorphism. It is also an embedding.
\begin{proposition}\label{SidonSet}
The set $\{g_1,\ldots,g_v\}$ is a Sidon set of the second kind.
\end{proposition}
\begin{proof}
We prove that $g_ig_j^{-1}\neq g_kg_l^{-1}$ for any $|\{i,j,k,l\}|\geq 3$. Suppose we have a counter-example.

Since $|\{i,j,k,l\}|\geq 3$, there is some index that appears only once. Without loss of generality, suppose $i$ appears only once. Let $a$ be such that $(g_1[a],\ldots,g_v[a])$ has all coordinates equal to $0_G$ except for $g_i[a]=g\neq 0_G$. Now by assumption $g_ig_j^{-1}= g_kg_l^{-1}$ which implies that $g_i[a]g_j[a]^{-1}= g_k[a]g_l[a]^{-1}$. But this simplifies to $g=0_G$, contradicting our selection of $a$.
\qed
\end{proof}
\noindent Proposition \ref{SidonSet} immediately implies that $X$ embeds in $G(X)$. We refer to the embedding $h$ as the \emph{canonical embedding} of $X$ in $G(X)$.

Free Cayley graphs were originally introduced by Naserasr and Tardif (unpublished) who used them to obtain lower bounds on the chromatic numbers of Cayley graphs. They also appear in Beaudou et al. \cite{Beaudou}. Naserasr and Tardif showed that homomorphisms from a graph $X$ to Cayley graphs for a group $G$ factor using the canonical embedding of $X$ in $G(X)$.
\begin{theorem}[Naserasr and Tardif]\label{NasTar}
Let $X$ be a graph and $W$ be a Cayley graph for the group $G$. If $f:V(X)\rightarrow V(W)$ is a homomorphism of $X$ into $W$, then $f=\psi\circ h$ where $h$ is the canonical embedding of $X$ in $G(X)$, and $\psi:\mathcal{F}_v(G)\rightarrow G$ is a group homomorphism.
\end{theorem}
\noindent In particular, Theorem \ref{NasTar} shows that any graph $Y$ constructed from a function $h$ as in Section \ref{Embeddings} is a homomorphic image of the free Cayley graph $G(X)$.

\section{Free Cayley Graphs for the Groups $\ints_p$}\label{AuxiliaryGraphs}

When $G$ is finite Abelian, the free Cayley graph $G(X)$ has a simpler description.

\begin{lemma}\label{FreeGroupA}
Let $G$ be a finite Abelian group with exponent $m$. Then $\mathcal{F}_k(G)\cong\ints_m^k$.
\end{lemma}
\begin{proof}
Since $G$ is Abelian, for any $f\in \mathcal{F}_k(G)$, we can write $f$ in terms of the canonical generators as
\[
f = \sum_{i=1}^k\alpha_ig_i,
\]
where the $\alpha_i$ are non-negative integers.

Moreover, every $f$ has a unique expression of this form where $0\leq\alpha_i<m$ for each $\alpha_i$. To see that this is the case, let $g\in G$ be an element with order $m$. By definition, for each $1\leq i\leq k$ there is some index $j_i$ so that
\[
(0_G,\ldots,g,\ldots,0_G)=(g_1[j_i],\ldots,g_k[j_i])
\]
(where the $g$ appears in the $i$th coordinate). Now $f[i]=\alpha_ig$ implies that there is a unique $0\leq\alpha_i<m$ that satisfies this equation. This is true for each $1\leq i\leq k$.

Define $h:\mathcal{F}_k(G)\rightarrow\ints_m^k$ by
\[
h(f)=h\left(\sum_{i=1}^k\alpha_ig_i\right)=\sum_{i=1}^m\alpha_ie_i,
\]
where each $0\leq\alpha_i<m$, and $e_i$ is the element of $\ints_m^k$ with $1$ in the $i$th coordinate, and all other coordinates equal to $0$.

Since the expression of $f$ is unique, $h$ is invertible, and hence gives a bijection between $\mathcal{F}_k(G)$ and $\ints_m^k$. It is easy to see that $h$ is a homomorphism, so $h$ is an isomorphism.\qed
\end{proof}

Note that the isomorphism $h$ gives an isomorphism between the graph $G(X)$ and the graph $X(\ints_m^v,\mathcal{C})$ where
\[
\mathcal{C} = \{e_i-e_j\,:\,\text{$i,j$ adjacent in $X$}\}.
\]

\section{Cliques in $X$ and $G(X)$}\label{CliquePrelim}

Recall the construction from Section \ref{Embeddings}. We have a graph $X$, a finite group $G$, and an assignment $h(i)=g_i$ of vertices to group elements. This gives a Cayley graph $Y=X(G,\mathcal{C})$ where
\[
\mathcal{C} = \{g_ig_j^{-1}\,:\,\text{$i,j$ adjacent in $X$}\}.
\]
The condition that the 2-sums of elements in $\{g_1,\ldots,g_v\}$ are distinct allows us to conclude that $X$ embeds in $Y$. If we strengthen this condition and require that the 3-sums are distinct, then $X$ and $Y$ share more structure.

In \cite{Codenotti}, Codenotti et al. show that if $G=\ints_m$, and the 3-sums $g_i+g_j+g_k$ of elements in $\{g_1,\ldots,g_v\}$ are distinct, then $\omega(X)=\omega(Y)$ (note that summands in the 3-sums may be repeated). Their proof can easily be adapted to apply to any finite Abelian group $G$. We give that adaptation here. Note that for $S\subseteq G$, if the the 3-sums of $S$ are distinct, then the 2-sums are also distinct. Also note that 3-sum distinct sets (with more than one element) cannot exist in groups of exponent at most three.
\begin{lemma}\label{AbGroup}
Let $G$ be a finite Abelian group, $X$ be a graph on $\{1,\ldots,v\}$ and $\{g_1,\ldots,g_v\}\subseteq G$ be a set of elements whose 3-sums are distinct. Let $Y$ be constructed from $X$ as above. Then $\omega(X)=\omega(Y)$.
\end{lemma}
\begin{proof}
Since the 3-sums of elements of $\{g_1,\ldots,g_v\}$ are distinct, $X$ embeds in $Y$ and $\omega(Y)\geq\omega(X)$. So it suffices to show that if $T$ is a clique in $Y$, then we can find a clique $S$ in $X$ with $|S|=|T|$. Since $Y$ is vertex transitive, we assume that $0_G\in T$ (so $T\subseteq\mathcal{C}\cup\{0_G\}$).

Consider $g_a-g_b$ adjacent to $g_c-g_d$ in the neighbourhood of $0_G$. Since these vertices are adjacent, there is some $g_e-g_f\in\mathcal{C}$ so that
\[
(g_a-g_b)+(g_e-g_f)=g_c-g_d.
\]
Rearranging we see that
\[
g_a+g_e+g_d=g_c+g_b+g_f,
\]
and since the 3-sums of the $g_i$ are distinct, $\{a,d,e\}=\{b,c,f\}$ as multisets. We also have that $a\neq b$, $c\neq d$ and $e\neq f$, as we started with vertices in $\mathcal{C}$. So there are two possibilities: either $(a,d,e)=(c,f,b)$, or $(a,d,e)=(f,b,c)$. In the first case we have that $g_a-g_b$ is adjacent to $g_a-g_d$ by $g_b-g_d\in\mathcal{C}$. We see that $abd$ forms a triangle in $X$. Likewise, the second case gives triangle $abc$ in $X$. So every triangle in $Y$ containing $0_G$ corresponds to a triangle in $X$.

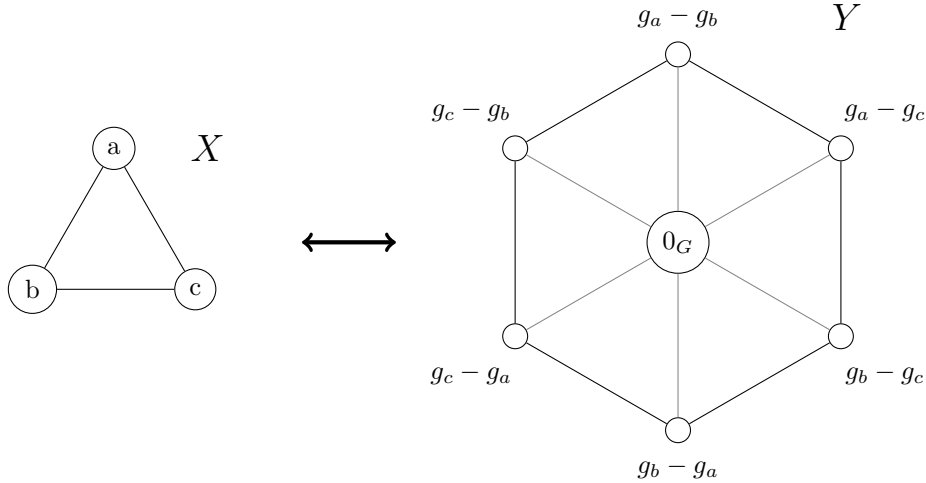
\begin{figure}[h]\label{figure}
\caption{Triangles in $X$ correspond to induced 6-cycles in $Y[0_G]$.}
\centering{
\begin{tikzpicture}[scale=2.5]
\node (label1) at (-2.5,0.5) [circle] {\Large{$X$}};
\node (x1) at (-3,0.5) [circle,draw] {a};
\node (x2) at (-3.433,-0.25) [circle,draw] {b};
\node (x3) at (-2.567,-0.25) [circle,draw] {c};
\draw (x1) edge (x2);
\draw (x2) edge (x3);
\draw (x3) edge (x1);
\draw [ultra thick,<->] (-2,0) -- (-1.5,0);
\node (label) at (0.9,1.2) [circle] {\Large{$Y$}};
\node (0) at (0,0) [circle,draw] {$0_G$};
\node (6) at (-0.866,0.5) [circle,draw] {};
\node (5) at (-0.866,-0.5) [circle,draw] {};
\node (4) at (0,-1) [circle,draw] {};
\node (3) at (0.866,-0.5) [circle,draw] {};
\node (2) at (0.866,0.5) [circle,draw] {};
\node (1) at (0,1) [circle,draw] {};
\draw (1) edge (2);
\draw (2) edge (3);
\draw (3) edge (4);
\draw (4) edge (5);
\draw (5) edge (6);
\draw (6) edge (1);
\draw (0) edge[gray,very thin] (1);
\draw (0) edge[gray,very thin] (2);
\draw (0) edge[gray,very thin] (3);
\draw (0) edge[gray,very thin] (4);
\draw (0) edge[gray,very thin] (5);
\draw (0) edge[gray,very thin] (6);
\node (l6) at (-1.1,0.7) [circle] {$g_c-g_b$};
\node (l5) at (-1.1,-0.7) [circle] {$g_c-g_a$};
\node (l4) at (0,-1.2) [circle] {$g_b-g_a$};
\node (l3) at (1.1,-0.7) [circle] {$g_b-g_c$};
\node (l2) at (1.1,0.7) [circle] {$g_a-g_c$};
\node (l1) at (0,1.2) [circle] {$g_a-g_b$};
\end{tikzpicture}
}
\end{figure}

Now let $abc$ be a triangle in $X$. The edges $ab,bc,ac$ give six distinct vertices
\[
\pm(g_a-g_b),\ \ \pm(g_b-g_c),\ \ \pm(g_a-g_c)
\]
in the neighbourhood of $0_G$ (as the 2-sums are distinct). Each edge between these vertices corresponds to an assignment of signs that makes the equation
\[
\pm(g_a-g_b)\pm(g_b-g_c)=\pm(g_a-g_c)
\]
valid. This follows as any connection set element involving $g_i\notin\{g_a,g_b,g_c\}$ cannot connect two of the six vertices by our initial argument. Every valid assignment of signs to the above equation will correspond to three edges, by interpreting each pair of $\pm(g_a-g_b),\pm(g_b-g_c)$ and $\pm(g_a=g_c)$ as the vertices.

For example, consider the equation
\[
(g_a-g_b)+(g_b-g_c)=-(g_a-g_c).
\]
Rearranging the terms, we have that $2g_a=2g_c$, which implies that $a=c$ (as the 2-sums are distinct). This contradicts the fact that we started with a triangle $abc$.

By considering every possible assignment of signs, we see that the only valid assignments are
\[
(g_a-g_b)+(g_b-g_c) = (g_a-g_c)
\]
and
\[
-(g_a-g_b)-(g_b-g_c) = -(g_a-g_c).
\]
These correspond to the edges
\[
\{g_a-g_b,g_a-g_c\},\ \{g_b-g_c,g_a-g_c\},\ \{g_a-g_b,g_c-g_b\}
\]
and
\[
\{g_b-g_a,g_c-g_a\},\ \{g_c-g_b,g_c-g_a\},\ \{g_b-g_a,g_b-g_c\}
\]
respectively. These six vertices and six edges form an induced 6-cycle in the neighbourhood of $0_G$.

Finally, consider the elements of $T$. Suppose $t_a,t_b\in T\setminus\{0_G\}$ where $a,b\in E(X)$ are the edges corresponding to $t_a,t_b$ respectively. Since $0_G,t_a,t_b$ is a triangle in $Y$ containing $0_G$, we have that $a$ and $b$ are edges in a triangle in $X$. Let $p\in V(X)$ be the vertex incident with both $a$ and $b$. By the same argument, for $t_b,t_c\in T\setminus\{0_G\}$ there is a vertex $q$ incident to both $b$ and $c$. Suppose $q\neq p$ (so $q$ is the other endpoint of $b$). Since $a$ and $c$ are also incident with some vertex of $X$, we have that $a,b,c$ are the edges of a triangle in $X$. Thus the vertices of $T\setminus\{0_G\}$ corresponding the edges of the triangle containing $a,b,c$ are part of an induced $6$-cycle in $Y$. But $t_a,t_b,t_c$ form a triangle in $Y$, a contradiction. Therefore every $t_a\in T\setminus\{0_G\}$ corresponds to an edge $a$ in $X$ incident with $p$.

Without loss of generality each $t_a\in T\setminus\{0_G\}$ has the form $t_a=g_p-g_x$ where $a=\{p,x\}\in E(X)$. Let $S=\{p\}\cup\{x\,:\,g_p-g_x\in T\}$. Then $S$ is a clique in $X$ with $|S|=|T|$.\qed
\end{proof}

Note that if $S$ is a clique in $X$ and $i\in S$, then $T=\{g_j-g_i\,:\,j\in S\}$ is a clique in $Y$ containing $0_G$ with $|T|=|S|$. So Lemma \ref{AbGroup} gives a method for constructing cliques in $X$ from cliques in $Y$, and vice versa. 

Free Cayley graphs for finite Abelian groups with exponent at least four are examples of graphs to which we can apply Lemma \ref{AbGroup} directly.
\begin{proposition}\label{3sum}
Let $G$ be a finite Abelian group with exponent $m\geq 4$, and $\{g_1,\ldots,g_v\}$ be the canonical generators for $\mathcal{F}_v(G)$. Then the 3-sums of the elements of $\{g_1,\ldots,g_v\}$ are all distinct.
\end{proposition}
\begin{proof}
Suppose $g_a+g_b+g_c=g_d+g_e+g_f$. Let $g\in G$ be an element of order at least $4$. Since every $v$-tuple of elements of $G$ appears as some $(g_1[i],\ldots,g_v[i])$, there is some $i$ so that $g_a[i]=g$, and $g_j[i]=0_G$ for all $j\neq a$. This implies that $a$ appears the same number of times in each of the multi-sets $\{a,b,c\}$ and $\{d,e,f\}$. Repeating this argument for $b$ and $c$ shows that the multisets are equal.\qed
\end{proof}
\noindent Since distinctness of 3-sums implies distinctness of 2-sums, $|\mathcal{C}|=2|E(X)|$, as $g_i-g_j$ is distinct for each arc $ij$ of $X$. Moreover, ${\bf 0}\notin\mathcal{C}$, so $G(X)$ is loopless. Using Lemma \ref{AbGroup} we immediately have the following corollary.
\begin{corollary}\label{Exp4}
If $G$ is a finite Abelian group with exponent $m\geq 4$, then $\omega(X)=\omega(G(X))$.\qed
\end{corollary}
\noindent Again, from Lemma \ref{AbGroup}, we also have a method for constructing cliques in $X$ from cliques in $G(X)$, and vice versa.

\section{The Groups $\ints_p$}\label{Zp}

Towards proving our main result, we consider the groups $G=\ints_p$ for $p$ prime. We start by showing that $\omega(X)$ is easily derived from $\omega(\ints_p(X))$. If $p\geq 5$, then $\ints_p$ has exponent $p>4$, and from Corollary \ref{Exp4} we have that $\omega(X)=\omega(\ints_p(X))$. The cases $p=2$ and $p=3$ require more care. Our arguments from Section \ref{CliquePrelim} do not apply directly, but the picture is similar. We begin with $p=2$.

Let $X$ be a graph, $G=\ints_2$ and $G(X)$ be the free Cayley graph. Note that in general, the $3$-sums of the canonical generators $g_i$ will not be distinct. For instance, $2g_i=2g_j$ for any $i$ and $j$. However, if we add the restriction that the summands are distinct, then the 3-sums are distinct.
\begin{proposition}\label{Z2Sums}
If $p=2$, the $2$-sums $g_i+g_j$ for distinct $1\leq i,j\leq v$ are distinct. The $3$-sums $g_i+g_j+g_k$ where $|\{i,j,k\}|=1,3$ are distinct.
\end{proposition}
\begin{proof}
Suppose that $1\leq i,j,k,l\leq v$ are such that $i\neq j$ and $k\neq l$, and
\[
g_i+g_j=g_k+g_l.
\]
If we suppose that $\{i,j\}\neq \{k,l\}$, then we can assume that $l\notin\{i,j,k\}$. Then there is some $1\leq x\leq 2^v$ such that $g_i[x]=g_j[x]=g_k[x]=0$, and $g_l[x]=1$. This is a contradiction, so we conclude that $\{i,j\}=\{k,l\}$.

Now suppose we have $1\leq i,j,k\leq v$ and $1\leq a,b,c\leq v$ such that
\[
g_i+g_j+g_k=g_a+g_b+g_c.
\]
We consider three cases. First, if $|\{i,j,k\}|=|\{a,b,c\}|=1$, then $3g_i=3g_a$ and $g_i=g_a$. So $\{i,j,k\}=\{a,b,c\}$.

Second, let $|\{i,j,k\}|=3$ and $|\{a,b,c\}|=1$. Then we have
\[
g_i+g_j+g_k=3g_a,
\]
which implies
\[
g_i+g_j=g_a+g_k.
\]
Now by our initial argument, $\{i,j\}=\{a,k\}$. But this implies that $|\{i,j,k\}|<3$, a contradiction.

Finally, let $|\{i,j,k\}|=|\{a,b,c\}|=3$. Suppose that $\{i,j,k\}\neq\{a,b,c\}$. This implies that, without loss of generality, $c\notin\{i,j,k\}$. So there is some $1\leq x\leq 2^v$ such that $g_i[x]=g_j[x]=g_k[x]=g_a[x]=g_b[x]=0$ and $g_c[x]=1$. Thus we have a contradiction.\qed
\end{proof}

Since each $g_i$ has order $2$ in $\mathcal{F}_v(\ints_2)$, $g_i=-g_i$ for all $1\leq i\leq v$. Also
\[
g_i-g_j=g_i+g_j=-g_i+g_j,
\]
and our connection set $\mathcal{C}$ is
\[
\mathcal{C}=\{g_i+g_j\,:\,\text{$i,j$ adjacent in $X$}\}.
\]
So $\ints_2(X)$ is $|E(X)|$-regular (as opposed to $2|E(X)|$-regular). We show that the properties in Proposition \ref{Z2Sums} are enough to guarantee that $\omega(\ints_2(X))$ and $\omega(X)$ are closely related.

Unlike the case where the exponent of $G$ is at least four, we will not be able to conclude that $\omega(\ints_2(X))=\omega(X)$ for all graphs $X$. Since $\ints_2(X)$ is a cubelike graph, $\omega(\ints_2(X))\neq 3$ (in a cubelike graph, every triangle is contained in a copy of $K_4$). Thus if $\omega(X)=3$, we won't have $\omega(\ints_2(X))=\omega(X)$. However, we can show that this is the only problematic case for $p=2$. We start with a simple observation.

\begin{proposition}\label{CommonGen}
Let $k\geq 4$. Suppose ${\bf 0},h_1,\ldots,h_k$ is a clique in $\ints_2(X)$ where $h_i=g_i^{(1)}+g_i^{(2)}$. Then there is some $g\in\{g_i\,:\,1\leq i\leq v\}$ so that for each $1\leq i \leq k$, there is $j\in\{1,2\}$ with $g=g_i^{(j)}$.
\end{proposition}
\begin{proof}
Since $h_i$ and $h_j$ are adjacent for each $1\leq i,j\leq k$, we have
\[
g_i^{(1)}+g_i^{(2)}+g_s+g_t=g_j^{(1)}+g_j^{(2)}
\]
for some $g_s+g_t\in \mathcal{C}$. Note that $g_i^{(1)}\neq g_i^{(2)}$, $g_j^{(1)}\neq g_j^{(2)}$ and $g_s\neq g_t$. Also, $\{g_i^{(1)},g_i^{(2)}\}\neq\{g_s,g_t\}$ and $\{g_j^{(1)},g_j^{(2)}\}\neq\{g_s,g_t\}$. Thus we can assume that $g_s\notin\{g_i^{(1)},g_i^{(2)}\}$ and $g_t\notin\{g_j^{(1)},g_j^{(2)}\}$. Rearranging,
\[
g_i^{(1)}+g_i^{(2)}+g_s=g_j^{(1)}+g_j^{(2)}+g_t,
\]
and from Proposition \ref{Z2Sums} we see that we must have either
\[
g_i^{(1)}\in\{g_j^{(1)},g_j^{(2)}\}\ \ \text{or}\ \ g_i^{(2)}\in\{g_j^{(1)},g_j^{(2)}\}.
\]
Thus
\[
|\{g_i^{(1)},g_i^{(2)}\}\cap\{g_j^{(1)},g_j^{(2)}\}|=1
\]
for all $h_i$ and $h_j$.

Suppose that no such element $g$ exists. Consider $h_1=g_1^{(1)}+g_1^{(2)}$. For each $h_i\neq h_1$ we have that either
\[
g_1^{(1)}\in\{g_i^{(1)},g_i^{(2)}\}\ \ \text{or}\ \ g_1^{(2)}\in\{g_i^{(1)},g_i^{(2)}\}.
\]
Let $S_1$ be the subset of $\{h_2,\ldots,h_k\}$ so that each $h_i\in S_1$ is of the form $g_1^{(1)}+g_i$, and let $S_2$ be defined analogously. Since $k\geq 4$, one of $S_1,S_2$ is not a singleton. Without loss of generality let $|S_1|>1$.

Since $g$ does not exist, $S_2$ is non-empty. We have $h_i,h_j\in S_1$ and $h_a\in S_2$. So
\[
h_i=g_1^{(1)}+g_i,\ \ h_j=g_1^{(1)}+g_j,\ \ \text{and}\ \ h_a=g_1^{(2)}+g_a
\]
where $g_a\neq g_1^{(1)}$ and $g_i,g_j\neq g_1^{(2)}$. Thus only one of $h_i$ and $h_j$ can be adjacent to $h_a$, a contradiction.\qed
\end{proof}

\begin{lemma}\label{Z2Cliques}
If $\omega(X)=3$, then $\omega(\ints_2(X))=4$. Otherwise, $\omega(\ints_2(X))=\omega(X)$.
\end{lemma}
\begin{proof}
Suppose that $S$ is a clique in $X$. Fix $i\in S$ and let
\[
T=\{g_i+g_j\,:\,j\in S\}.
\]
Then $T$ is a subset of the vertices of $\ints_2(X)$ containing ${\bf 0}$. Since $i$ is adjacent to every other vertex in $S$, we see $g_i+g_j\in \mathcal{C}$ for all $j\in S\setminus\{i\}$. Thus ${\bf 0}$ is adjacent to every $g_i+g_j\in T$ with $i\neq j$. Also, if $g_i+g_j\in T$ and $g_i+g_k\in T$, then since $j,k\in S$ we have that $g_j+g_k\in \mathcal{C}$, and so $g_i+g_j$ is adjacent to $g_i+g_k$. Therefore $T$ is a clique in $\ints_2(X)$ with $|T|=|S|$, and $\omega(X)\leq\omega(\ints_2(X))$.

Now suppose $S$ is a clique in $\ints_2(X)$. Assume that $|S|\geq 5$, and without loss of generality that ${\bf 0}\in S$. By Proposition \ref{CommonGen} there is a vertex $i$ of $X$ so that every element of $S\setminus\{{\bf 0}\}$ is of the form $g_i+g_j$. Thus the vertices
\[
\{i\}\cup\{j\,:\,g_i+g_j\in S\}
\]
form a clique in $X$ of size $|S|$, and $\omega(X)=\omega(\ints_2(X))$.

If $\omega(X)=3$, then for any triangle $\{a,b,c\}$ in $X$, the vertices
\[
\{{\bf 0},g_a+g_b,g_a+g_c,g_b+g_c\}
\]
form a clique in $\ints_2(X)$. Since $\omega(\ints_2(X))\geq 5$ implies $\omega(X)=\omega(\ints_2(X))$, we have $\omega(\ints_2(X))=4$. This also shows that if $\omega(X)=4$, then $\omega(\ints_2(X))=4$.

If $\omega(X)=2$, then $X$ has edges but is triangle-free. Thus the neighbourhood of ${\bf 0}$ in $\ints_2(X)$ contains no edges and so $\omega(\ints_2(X))=2$. The case $\omega(X)=1$ is trivial.\qed
\end{proof}

Finally, we turn to the case where $p=3$. Again, we cannot apply our reasoning from Section \ref{CliquePrelim} directly. If $p=3$, then our groups have exponent three, and the $3$-sums of the generators will not all be distinct. However, we will again be able to show that a large enough subset of the $3$-sums will be distinct.

For the $p=3$ case, our general strategy is the same as for groups with exponent at least four. Our arguments are very similar to the arguments from Section \ref{CliquePrelim} (unlike the $p=2$ case). So we simply present the general strategy, and necessary lemmas, and leave the details to the interested reader (a complete proof can be found in Chapter 3 of \cite{Rooney}).

When $p=3$, the $2$-sums of generators will be distinct. So, given a graph $X$, the graph $\ints_3(X)$ is $2|E(X)|$-regular as usual. In this case, the assignment almost gives distinct $3$-sums. The only problem is the fact that $3g_i={\bf 0}$ for all $i$.

\begin{proposition}\label{Z3Sums}
For $p=3$, if $g_i+g_j+g_k=g_r+g_s+g_t$, then either $\{i,j,k\}=\{r,s,t\}$ as multisets, or $|\{i,j,k\}|=|\{r,s,t\}|=1$.
\end{proposition}

From this proposition we can prove an analogue to Lemma \ref{AbGroup}. The proof follows the proof of Lemma \ref{AbGroup} very closely, using Proposition \ref{Z3Sums} in place of 3-sum distinctness.
\begin{lemma}\label{m3CliqueX}
If $\omega(\ints_3(X))\geq 3$ , then $\omega(X)=\omega(\ints_3(X))$.
\end{lemma}

Lemma \ref{m3CliqueX} also gives the usual correspondence between the cliques of $X$ and the cliques of $\ints_3(X)$. The full relationship between $\omega(X)$ and $\omega(\ints_3(X))$ now follows easily.
\begin{lemma}\label{Z3Cliques}
If $\omega(X)=2$, then $\omega(\ints_3(X))=3$. Otherwise, $\omega(\ints_3(X))=\omega(X)$.
\end{lemma}

Corollary \ref{Exp4} and Lemmas \ref{Z2Cliques} and \ref{Z3Cliques} give the precise relation between $\omega(X)$ and $\omega(\ints_p(X))$ when $p$ is a prime. As in Section \ref{FreeCayleyGraphs}, the proofs also give an efficient method for producing a maximum clique in $X$ given a maximum clique in $\ints_p(X)$, and vice versa.

\section{Quotients of $\ints_p(X)$}\label{LinearCodes}

In order to reduce computing clique number for general graphs to computing clique number for Cayley graphs for the groups $\ints_p^n$, we need to be able to construct an auxiliary graph $X(\ints_p^n,\mathcal{C})$ from a graph $X$ so that: we can find a maximum clique in $X$ efficiently given a maximum clique in $X(\ints_p^n,\mathcal{C})$; and, the size of $X(\ints_p^n,\mathcal{C})$ is bounded by a polynomial in the size of $X$. We showed that $\ints_p(X)$ is an auxiliary graph with the first property. However, as we noted in Section \ref{AuxiliaryGraphs}, $|\ints_p(X)|=p^v$ is exponential in $v$, the size of the input graph. We fix that problem by using linear codes to construct a quotient of $\ints_p(X)$.

Recall that $\ints_p(X)$ is isomorphic to the graph $X(\ints_p^v,\mathcal{C})$ where
\[
\mathcal{C} = \{e_i-e_j\,:\,\text{$i,j$ adjacent in $X$}\}.
\]
Since we will be focussing on this graph, we will take $\ints_p(X)$ to refer to the graph $X(\ints_p^v,\mathcal{C})$, instead of the free Cayley graph.

The vertex set of $\ints_p(X)$ is a vector space, as well as a group. The subspaces of $\ints_p^v$ are linear codes, and following the material in Section \ref{QuotientGraphs} we can use these codes to construct quotients of $\ints_p(X)$. Our goal is to find a code $D$ in $\ints_p^v$ so that $\omega(X)$ can easily be computed from the quotient graph $\ints_p(X)_D$. We achieve this by constraining the distance $d$ of $D$.

\begin{proposition}\label{d3Coclique}
If $D$ is a code in $\ints_p^v$ with distance $d\geq 3$, then $D$ is a coclique in $\ints_p(X)$.
\end{proposition}
\begin{proof}
From Proposition \ref{CosetPartitionEquitable} we have that $\Pi_D$ is an equitable partition of $\ints_p(X)$. Thus $\ints_p(X)[D]$ is a regular subgraph of $\ints_p(X)$. From Proposition \ref{CosetPartitionEquitable} we have that $\ints_p(X)[D]$ is $|D\cap \mathcal{C}|$-regular.

Now consider any element of $|D\cap \mathcal{C}|$. Since $D$ has distance $d\geq 3$, all non-zero elements of $D$ have weight at least $3$. But $e_i-e_j$ has weight $2$ for all $i\neq j$. Thus $|D\cap \mathcal{C}|=0$, and $D$ is a coclique in $\ints_p(X)$.\qed
\end{proof}
\noindent It follows from Propositions \ref{CosetPartitionEquitable} and \ref{d3Coclique} that $\Pi_D$ gives a partition of $\ints_p(X)$ into cocliques. This immediately implies that any clique in $\ints_p(X)$ gives a corresponding clique in $\ints_p(X)_D$ of equal size. So we have $\omega(\ints_p(X)_D)\geq\omega(\ints_p(X))$. We also have the following immediate corollary.
\begin{corollary}\label{CQuotient}
If $d\geq 3$, then $\ints_p(X)_D$ is a Cayley graph for the group $\ints_p^v/D$ with connection set
\[
\mathcal{C}'=\{(D+e_i)-(D+e_j)\,:\,\text{$i$ is adjacent to $j$ in $X$}\},
\]
and the map $f(g)=D+g$ gives a bijection between $\mathcal{C}$ and $\mathcal{C}'$.\qed
\end{corollary}

\begin{proposition}\label{d5InducedX}
If $D$ has distance $d\geq 5$, then $\ints_p(X)_D$ contains an induced copy of $X$.
\end{proposition}
\begin{proof}
Since $d\geq 5$, there is no $i$ so that $e_i\in D$. Moreover, if $D+g$ is a coset of $D$, then for indices $i\neq j$ we cannot have both $e_i,e_j\in D+g$. This follows as otherwise there are $\alpha,\beta\in D$ so that $\alpha+g=e_i$ and $\beta+g=e_j$. Thus
\[
e_i-e_j=\alpha-\beta\in D.
\]
However, $e_i-e_j$ has weight $2<5$, contradicting the distance of $D$. Thus $D+e_i$ is a vertex of $\ints_p(X)_D$ for each $1\leq i\leq v$. We show that the vertices
\[
\{D+e_i\,:\,1\leq i\leq v\}
\]
give an induced copy of $X$ in $\ints_p(X)_D$.

Consider adjacent vertices $i,j$ in $X$. We have that $e_i-e_j\in \mathcal{C}$, and $e_i$ and $e_j$ are connected by an edge in $\ints_p(X)$. We also have that $e_i\in D+e_i$ and $e_j\in D+e_j$, so $D+e_i$ and $D+e_j$ are connected by an edge in $\ints_p(X)_D$. So $X$ is a subgraph of $\ints_p(X)_D$.

Now suppose that $i,j$ are non-adjacent vertices of $X$. Suppose that $D+e_i$ and $D+e_j$ are connected by an edge in $\ints_p(X)_D$. Then we have $\alpha,\beta\in D$ and $e_a-e_b\in \mathcal{C}$ so that
\[
\alpha+e_i+e_a-e_b=\beta+e_j,
\]
or
\[
e_i+e_a-e_b-e_j=\beta-\alpha\in D.
\]
But the weight of the left-hand side of the equation is at most $4$ and $d\geq 5$ so we have a contradiction.\qed
\end{proof}
\noindent Proposition \ref{d5InducedX} immediately implies that $\omega(X)\leq\omega(\ints_p(X)_D)$.

Finally, if we increase the distance of $D$ again, we can show that $\ints_p(X)_D$ will have the same maximum clique size as $X$, with the usual exceptions for $p=2,3$. Our approach is to show that the elements of $\mathcal{C}'$ in each case satisfy the same properties as those of $\mathcal{C}$ with respect to $2$-sums and $3$-sums. As a result, the proofs in Sections \ref{CliquePrelim} and \ref{Zp} (Lemmas \ref{AbGroup}, \ref{Z2Cliques} and \ref{Z3Cliques} in particular) will apply unchanged.
\begin{lemma}\label{CliquesGammaD}
Let $D$ be a code with distance $d\geq 7$. If $p\geq 5$, then $\omega(\ints_p(X)_D)=\omega(X)$. If $p=3$, then $\omega(\ints_p(X)_D)=\omega(X)$ unless $\omega(X)=2$, in which case $\omega(\ints_p(X)_D)=3$. If $p=2$, then $\omega(\ints_p(X)_D)=\omega(X)$ unless $\omega(X)=3$, in which case $\omega(\ints_p(X)_D)=4$.
\end{lemma}
\begin{proof}
We begin by assuming that $p\geq 4$. In this case we show that the $3$-sums of elements of $\mathcal{C}'$ are all distinct. Let $D+(e_i+e_j+e_k)$ and $D+(e_a+e_b+e_c)$ be cosets of $D$ for any $\{i,j,k\}\neq\{a,b,c\}$. Suppose that
\[
D+(e_i+e_j+e_k)=D+(e_a+e_b+e_c).
\]
Then we have $\alpha,\beta\in D$ so that
\[
\alpha+e_i+e_j+e_k=\beta+e_a+e_b+e_c,
\]
and as a result,
\[
e_i+e_j+e_k-e_a-e_b-e_c=\alpha-\beta\in D.
\]
This gives an immediate contradiction as the weight of the left-hand side of this equation is at most $6$ and at least $2$, while $d\geq 7$. Therefore we have that $\omega(\ints_p(X)_D)=\omega(X)$.

In the cases $p=2$ and $p=3$, we need to show that the $2$-sums and $3$-sums of the cosets of $D$ corresponding to the vectors $e_i$ satisfy Propositions \ref{Z2Sums} and \ref{Z3Sums}. The proofs follow similar reasoning as argument given above, so we omit the details.\qed
\end{proof}

To complete the construction of our auxiliary graph, it remains to show that we can find a code $D$ with $d\geq 7$, and with rank large enough so that $|\ints_p(X)_D|$ is polynomial in $|X|$.

\section{A Goppa Code}\label{AGoppaCode}

Consider the Goppa polynomial $g(x)=x^6$. For any $m\geq 1$, $g(x)$ is a polynomial with coefficients in $\GF(p^m)$ and every element of $\GF(p^m)\setminus\{0\}$ is a non-root of $g(x)$. So we can let $L$ be any subset of non-zero elements of $\GF(p^m)$. Let $D(g,L)$ be the Goppa code constructed from $g(x)$ and $L$. From Section \ref{GoppaCodes} we have that $D(g,L)$ will have distance $d\geq 7$, block length $|L|$ and rank $k\geq |L|-6m$ where $|L|\leq p^m-1$.

We want the block length of $D(g,L)$ to be $v$, the number of vertices of $X$. We also want the rank of $D(g,L)$ to satisfy $p^{v-k}\leq f(v)$ for all $v\geq N$, where $f(x)$ is a polynomial and $N$ is some fixed integer. The block length of $D(g,L)$ is $|L|$, and $L$ can be any subset of $\GF(p^m)\setminus\{0\}$. So we are able to choose some such $L$ with $|L|=v$ provided $v\leq p^m-1$. We rearrange the constraint $p^{v-k}\leq f(v)$ as $k\geq v-\log_pf(v)$. In order to ensure this inequality is satisfied, we want to choose $m$ so that
\[
k\geq v-6m\geq v-\log_pf(v),
\]
or $m\leq\log_pf(v)/6$.

\begin{lemma}\label{7Goppa}
There is an integer $N$ so that for all $v\geq N$, we can choose $m$ to satisfy $v\leq p^m-1$ and $m\leq \log_pv^{12}/6$.
\end{lemma}
\begin{proof}
Note that $m\leq\log_pv^{12}/6$ implies $m\leq\log_pv^2$, or $p^m\leq v^2$. Also the condition $v\leq p^m-1$ is equivalent to $v<p^m$ as all the quantities are integers.

Take $N=p^2$. Now for any $v\geq N$, the interval $(\log_pv,2\log_pv]$ contains an integer, as $\log_pv\geq 2$. Choose $m$ to be the largest such integer.\qed
\end{proof}
\noindent Choose $m$ to be an integer in $(\log_pv,2\log_pv]$. We can take $L$ to be an arbitrary set of non-zero elements of $\GF(p)^m$ of size $v$, and the Goppa code $D(g,L)$ will have rank $k\geq v-\log_pv^{12}$. This shows that a suitable Goppa code always exists (and is easily specified). It remains to show that we can use $D=D(g,L)$ and construct $\ints_p(X)_D$ efficiently.

The construction we have outlined so far involves constructing $\ints_p(X)$, and then $\ints_p(X)_D$ as a quotient of $\ints_p(X)$. However, this involves constructing a graph with an exponential number of vertices. In order to get around this problem, we note that a Cayley graph is specified by its connection set. 

By Corollary \ref{CQuotient} we have that $\ints_p(X)_D$ is a Cayley graph $X(\ints_p^{v-k},\mathcal{C}')$, so we can construct $\ints_p(X)_D$ as follows. First we take the connection set $\mathcal{C}$ of $\ints_p(X)$ defined as usual (this set has size polynomial in $v$ as its size is a constant multiple of the number of edges of $X$). Then we will use the generator matrix of $D(g,L)$ to construct the connection set $\mathcal{C}'$ of $\ints_p(X)_D$ from $\mathcal{C}$ directly. As long as this can be done in polynomial time with a polynomial amount of space, we will have the desired construction.

We have already seen an explicit description of a parity check matrix $H$ for $D(g,L)$ in Section \ref{GoppaCodes}. Using $H$ we can recover a generator matrix $B$ for $D(g,L)$ so that $D(g,L)=\row(B)$. From the rows of $B$ we can find a basis $\{\beta_1,\ldots,\beta_k\}$ for $D(g,L)$ and extend this basis to a basis for $\ints_p^v$,
\[
\{\beta_1,\ldots,\beta_k,\beta_{k+1},\ldots,\beta_n\}.
\]
Now any $\alpha\in\ints_p^v$ can be written uniquely as
\[
\alpha=\sum_{i=1}^va_i\beta_i
\]
where the $a_i$ are elements of $\ints_p$. Furthermore, in the quotient space $\ints_p^v/D(g,L)$, the coset containing $\alpha$ is
\[
D+\left(\sum_{i=k+1}^va_i\beta_i\right).
\]
Thus the elements of the connection set $\mathcal{C}$ can be expressed using our basis, and we set
\[
\mathcal{C}'=\left\{\sum_{i=k+1}^va_i\beta_i\,:\,\sum_{i=1}^va_i\beta_i\in \mathcal{C}\right\}.
\]

\begin{lemma}\label{EfficientGoppa}
Let $p$ be a fixed prime. Given a graph $X$ with at least $p^2$ vertices, $\ints_p(X)_D$ can be constructed in polynomial time and space.
\end{lemma}
\begin{proof}
From Lemma \ref{7Goppa}, we choose $m$ to be the largest integer in $(\log_pv,2\log_pv]$. Construct the field $\GF(p^m)$ by finding an irreducible polynomial $f$ of degree $m$ over the field $\GF(p)$, and representing $\GF(p^m)$ as $\GF(p)[x]/\langle f(x)\rangle$ where $\langle f(x)\rangle$ is the ideal generated by $f(x)$. This can be done in time polynomial in $m$ \cite{Codenotti}, and hence in time polynomial in $v$.

We choose a subset $L\subseteq \GF(p^m)\setminus\{0\}$ with $|L|=v$ as follows. Let $\alpha\in \GF(p^m)$ be a primitive element. We can find $\alpha$ by calculating $a^i$ for all $1\leq i\leq p^m-1$ and $a\in \GF(p^m)$. This involves checking at most $p^m\leq v^2$ elements, each of which requires at most $v^2$ multiplications in $\GF(p^m)$, so this can be accomplished in polynomial time.

Set $L=\{\alpha^i\,:1\leq i\leq v\}$. Set $g(x)=x^6$, and consider the Goppa code $D(g,L)$. We construct a check matrix $H$ for $D(g,L)$ as in Section \ref{GoppaCodes}. We set
\[
H'[i,j]=\alpha_j^ig(\alpha_j)^{-1}
\]
for $1\leq i\leq r$ and $1\leq j\leq v$. Since $H'$ is a $r\times v$ matrix, with $r<v$, this involves at most $v^2$ calculations, each of which involves $O(v^2)$ computations in $\GF(p^m)$. We obtain a $rm\times v$ matrix $H$ from $H'$ by replacing each entry of $H'$ with a vector in $\GF(p)^m$ corresponding to its entry in $\GF(p^m)$. Again this requires at most $v^2$ replacement operations, each of which takes time $O(v)$, as given a polynomial $\alpha\in \GF(p^m)$ (recall that we are using the construction $\GF(p^m)=\GF(p)[x]/\langle f(x)\rangle$) we replace $\alpha$ with the vector of its coefficients in $\GF(p)$.

From $H$ we construct a basis for $\ints_p^v$. We have that $H$ is a $rm\times v$ matrix whose rows span a space of dimension $v-k$, and whose null space has dimension $k$. We take $\{\beta_1,\ldots,\beta_k\}$ to be a basis for $\ns(H)=D(g,L)$, and $\{\beta_{k+1},\ldots,\beta_v\}$ to be a basis for $\row(H)$. Now $B=[\beta_1\ldots\beta_v]$ is a matrix whose columns are a basis for $\ints_p^v$. We can find $B$ by converting $H$ into reduced row-echelon form (in time polynomial in $v$, as $H$ has at most $v$ rows and columns).

For each $1\leq i\leq v$ let $e_i$ be the $i$th standard basis vector in $\ints_p^v$. Let
\[
\mathcal{C}=\{e_i-e_j\,:\text{$i$ is adjacent to $j$ in $X$}\},
\]
as usual. Each element of $\mathcal{C}$ can be uniquely expressed as a sum of columns of $B$. So we set
\[
\mathcal{C}'=\left\{\sum_{i=k+1}^va_i\beta_i\,:\,\sum_{i=1}^va_i\beta_i\in \mathcal{C}\right\}.
\]
For each $\alpha\in\ints_p^v$, to write $\alpha$ as a sum of columns of $B$, we solve the matrix equation $Bx=\alpha$. This can be done in polynomial time for each $\alpha\in \mathcal{C}$, so in total we solve $O(v^2)$ equations to find $\mathcal{C}'$.

From $\mathcal{C}'$ we construct the graph $\ints_p(X)_D=X(\ints_p^{v-k},\mathcal{C}')$. Recall that we chose $m$ so that $p^{v-k}$ is polynomial in $v$. So constructing $\ints_p(X)_D$ is done in polynomial time and space.\qed
\end{proof}

\section{Clique Number}\label{CliqueNumber}

In the preceding sections, we have given a construction of a Cayley graph $\ints_p(X)_D$ for a group $\ints_p^m$ from a graph $X$. We now show how to use this construction to prove our first main result.

\begin{theorem}\label{SpecialCase}
Let $p$ be a prime. Computing clique number is NP-Hard for the class of Cayley graphs for the groups $\ints_p^n$.
\end{theorem}
\begin{proof}
Assume that we are given an oracle $\Omega$ that computes the clique number of any graph $X(\ints_p^n,C)$ in time polynomial in $p^n$. We are given a graph $X$ on $v$ vertices.

Suppose $v<p^2$. In this case we simply solve for $\omega(X)$ exhaustively (since there are only finitely many graphs with less than $p^2$ vertices).

Assume $v\geq p^2$. By Lemma \ref{EfficientGoppa}, we can construct an auxiliary graph $\ints_p(X)_D=X(\ints_p^m,\mathcal{C}')$ from $X$ in polynomial time. By construction, the size of $\ints_p(X)_D$ is bounded polynomially in $v$. We use our clique number oracle $\Omega$ to compute $\omega(\ints_p(X)_D)$. By assumption, $\Omega$ runs in time polynomial in the size of the input graph, which is polynomial in $v$. So $\Omega$ returns $\omega(\ints_p(X)_D)$ in time polynomial in $v$.

Finally, we compute $\omega(X)$ from $\omega(\ints_p(X)_D)$. If $p\geq 4$, Lemma \ref{CliquesGammaD} shows that $\omega(X)=\omega(\ints_p(X)_D)$, so our computation takes constant time. If $p=2$, then Lemma \ref{CliquesGammaD} gives us that either $\omega(X)=\omega(\ints_p(X)_D)$, or that $\omega(\ints_p(X)_D)=4$ and $\omega(X)=3$ or $4$. If $\omega(\ints_p(X)_D)=4$, we check the $4$-subsets of $V(X)$ exhaustively for cliques to determine whether $\omega(X)=3$ or $4$. This takes $O(v^4)$, so the entire procedure runs in polynomial time. Likewise for $p=3$, Lemma \ref{CliquesGammaD} gives a polynomial time method to compute $\omega(X)$ from $\omega(\ints_p(X)_D)$.\qed
\end{proof}

As a final note, we point out that in the above proof, if our oracle $\Omega$ returns a maximum clique in the auxiliary graph $\ints_p(X)_D$, then the proofs of Lemmas \ref{AbGroup}, \ref{Z2Cliques} and \ref{Z3Cliques} give a method for finding a maximum clique in $X$ in polynomial time.

Theorem \ref{SpecialCase} easily generalizes to our full result.

\begin{theorem}\label{MainTheorem}
Let $G$ be a finite group. Computing clique number is NP-Hard for the class of Cayley graphs for the groups $G^n$.
\end{theorem}
\begin{proof}
As in the proof of Theorem \ref{SpecialCase} we give a polynomial time reduction from the clique number problem on the class of all graphs. We are given a graph $X$ on $v$ vertices, and want to construct an auxiliary graph that is a Cayley graph for a group of the form $G^m$. Let $p$ be a prime so that there is a subgroup $H$ of $G$ with $H\cong\ints_p$ (the existence of $p$ follows immediately from Cauchy's Theorem, see p. 10 in \cite{Isaacs}).

We construct $\ints_p(X)_D=X(\ints_p^m,\mathcal{C})$ as usual. Recall that our construction ensures that $p^m$ is polynomially bounded in $v$, the graph $\ints_p(X)_D$ can be constructed in time bounded by a polynomial in $v$, and $\omega(X)$ can be calculated from $\omega(\ints_p(X)_D)$ in time polynomial in $v$.

Since $\mathcal{C}\subseteq\ints_p^m$ the isomorphism between $\ints_p^m$ and $H^m$ maps $\mathcal{C}$ to $\mathcal{C}'\subseteq H^m$. Consider the graph $\Gamma=X(G^m,\mathcal{C}')$. Since $\mathcal{C}'\subseteq H$, the graph $\Gamma$ consists of $(|G|/|H|)^m$ isomorphic copies of $\ints_p(X)_D$ (i.e., we have one copy of $\ints_p(X)_D$ for each coset of $H^m$ in $G^m$). Therefore either $\omega(X)=\omega(\Gamma)$, or $p=2,3$ and we have the usual caveats.

Moreover, we know that $p^m\leq f(v)$ where $f(x)$ is some polynomial in $x$. There is some integer $\alpha$ so that $|G|\leq p^{\alpha}$. Thus
\[
(|G|/|H|)^m\leq (p^{\alpha-1})^m=(p^m)^{\alpha-1}\leq (f(v))^{\alpha-1}
\]
and the size of $\Gamma$ is bounded by a polynomial in $v$. This completes the proof.\qed
\end{proof}

\section{Chromatic Number}\label{ChromaticNumber}

Given Theorem \ref{MainTheorem}, it is natural to consider other hard problems on graphs, and ask whether those problems remain hard for the class of Cayley graphs for the groups $G^n$. Codenotti et al. \cite{Codenotti} prove that the chromatic number problem is NP-Hard for circulants. So we might hope to adapt their argument to prove that the chromatic number problem is NP-Hard for our class of Cayley graphs. However, a straightforward adaptation does not work. In this section we prove that computing chromatic number cannot be done in polynomial time for the class of Cayley graphs on the groups $G^n$ where $G$ is a fixed finite group.

We begin by restricting our consideration to the groups $\ints_p^n$ for some prime $p$ (as in the proof of Theorem \ref{MainTheorem}). Given a graph $X$ on $v$ vertices, we will use our previous construction to construct $\ints_p(X)_D$, with size $p^n$  bounded by a polynomial in $v$. Given a Cayley graph for $\ints_p^n$ we use the following construction to construct a Cayley graph for $\ints_p^m$ whose chromatic number is related to the clique number of the input graph.

Let $\Gamma=X(\ints_p^n,C)$ and $1\leq i\leq n$. Let $\Gamma_i$ be the graph obtained from $K_{p^i}\Box K_{p^n}$ by adding an edge between $(a,x)$ and $(b,y)$ if and only if $a\neq b$ and $x$ and $y$ are not adjacent in $\Gamma$ (where $a,b\in\ints_p^i$ and $x,y\in\ints_p^n$).

\begin{lemma}\label{AuxChi}
For each $1\leq i\leq n$, the graph $\Gamma_i$ is a Cayley graph for the group $\ints_p^{n+i}$ with $\alpha(\Gamma_i)=\min\{p^i,\omega(\Gamma)\}$. Moreover, $\chi(\Gamma_i)=p^n$ if and only if $\omega(\Gamma)\geq p^i$.
\end{lemma}
\begin{proof}
To prove that $\Gamma_i$ is a Cayley graph for the group $\ints_p^{n+i}$, we give a connection set $\mathcal{S}$ so that $\Gamma_i=X(\ints_p^{n+i},\mathcal{S})$. First we consider $\ints_p^{n+i}=\ints_p^i\times\ints_p^n$. Taking the elements
\[
\mathcal{K}=\{(x,{\bf 0_{p^n}})\,:\,x\in\ints_p^i\setminus\{{\bf 0_{p^i}}\}\}\cup\{({\bf 0_{p^i}},y)\,:\,y\in\ints_p^n\setminus\{{\bf 0_{p^n}}\}\}
\]
gives the connection set for $K_{p^i}\Box K_{p^n}$. We also need to have the edges connecting $(a,\alpha)$ and $(b,\beta)$ for all $a\neq b$ and $\alpha$ and $\beta$ non-adjacent in $\Gamma$. Since $\Gamma=X(\ints_p^n,C)$, the complement of $\Gamma$ is a Cayley graph for $\ints_p^n$ with connection set
\[
C'=\ints_p^n\setminus(C\cup\{{\bf 0_{p^n}}\}).
\]
Thus our desired connection set is
\[
\mathcal{S}=\mathcal{K}\cup\{(x,c)\,:\,x\in\ints_p^i\setminus\{{\bf 0_{p^i}}\}, c\in C'\}.
\]

Consider $\alpha(\Gamma_i)$. Since $\Gamma_i$ is vertex transitive, the clique-coclique bound (Corollary 4 in \cite{Cameron}) gives $\alpha(\Gamma_i)\omega(\Gamma_i)\leq p^{n+i}$. Since $\omega(\Gamma_i)\geq p^n$, we have that $\alpha(\Gamma_i)\leq p^i$. Moreover, if $S$ is a coclique in $\Gamma_i$, then $S$ contains at most one vertex from each copy of $K_{p^i}$ and at most one vertex from each copy of $K_{p^n}$. Thus if $(a,\alpha),(b,\beta)\in S$, then $a\neq b$ and $\alpha\neq \beta$. Finally, since $(a,\alpha),(b,\beta)\in S$ are non-adjacent, $\alpha$ and $\beta$ are adjacent in $\Gamma$. Therefore, $|S|\leq\omega(\Gamma)$. Using maximum cliques in $\Gamma$ we can construct cocliques in $\Gamma_i$ of size $\min\{p^i,\omega(\Gamma)\}$, thus $\alpha(\Gamma_i)=\min\{p^i,\omega(\Gamma)\}$.

Finally we show that $\chi(\Gamma_{p^i})=p^n$ if and only if $\omega(\Gamma)\geq p^i$. First, if $\omega(\Gamma)\geq p^i$, then $\alpha(\Gamma_i)=p^i$ and $\omega(\Gamma_i)=p^n$. Since $\alpha(\Gamma_i)\omega(\Gamma_i)=p^{n+i}$ we can partition the vertices of $\Gamma_i$ into $p^n$ cocliques of size $p^i$ using a clique of size $p^n$. This gives us a $p^n$-colouring of $\Gamma_i$ and proves that $\chi(\Gamma_{p^i})=p^n$.

Now suppose that $\omega(\Gamma)<p^i$. Then we have
\[
\chi(\Gamma_i)\geq \frac{p^{n+i}}{\alpha(\Gamma_i)}>\frac{p^{n+i}}{p^i}=p^n
\]
completing the proof.
\qed
\end{proof}

Suppose we have an oracle $\Omega$ that gives the chromatic number of Cayley graphs for $\ints_p^m$ in polynomial time. Given a graph $X$, take $\Gamma=\ints_p(X)_D$ from $X$ as usual ($\Gamma$ is a Cayley graph for $\ints_p^m$). Now for each $1\leq i\leq m$, we can construct the graph $\Gamma_i$. Since $\Gamma$ has size polynomial in the size of $X$, each $\Gamma_i$ has size polynomial in the size of $X$. Thus we can apply our oracle $\Omega$ and compute $\chi(\Gamma_i)$ for each $1\leq i\leq m$ in time polynomial in the size of $X$.

Now by Lemma \ref{AuxChi} we can find the value $y$ for which $\chi(\Gamma_y)=p^n$, and $\chi(\Gamma_i)>p^n$ for all $i>y$. This implies that $y=\lfloor \log_p\omega(X)\rfloor$; and using $\Omega$ we can compute $y$ in polynomial time. Since we can't compute $\omega(X)$ exactly, we can't conclude that our oracle $\Omega$ cannot exist directly from the fact that computing clique number is NP-Hard. However, not only is the clique number of a graph hard to compute, it is also hard to approximate.

H{\aa}stad \cite{Hastad} proved, under the assumption $\text{P}\neq\text{NP}$, that for any $\epsilon>0$, clique number cannot be efficiently approximated within $O(v^{1-\epsilon})$ (where $v$ is the size of the input graph). This means that there is no polynomial time algorithm that takes a graph $X$ on $v$ vertices and computes an output $\omega$ so that
\[
\frac{\omega(X)}{v^{1-\epsilon}}\leq \omega\leq \omega(X),
\]
for all $v$. Alternatively, we cannot have $\omega(X)/\omega\leq v^{1-\epsilon}$ for all $v$.

However, using the chromatic number oracle $\Omega$, the algorithm outlined above computes $\omega=p^y$ where $y=\lfloor \log_p\omega(X)\rfloor$ in polynomial time. Since $\omega\leq\omega(X)$, and $\omega(X)/\omega\leq p$, this contradicts H{\aa}stad's result, and implies that $\Omega$ cannot exist. Thus we have proved the following theorem.

\begin{theorem}\label{pChromatic}
Computing chromatic number cannot be done in polynomial time for the class of Cayley graphs for the groups $\ints_p^n$ where $p$ is a fixed prime (assuming $\text{P}\neq\text{NP}$).\qed
\end{theorem}
\noindent Theorem \ref{pChromatic} generalizes easily to the following theorem. The proof uses the same method we used to prove Theorem \ref{MainTheorem} from Theorem \ref{SpecialCase} verbatim.
\begin{theorem}\label{ChromCor}
Let $G$ be a finite group. Computing chromatic number cannot be done in polynomial time for Cayley graphs for the groups $G^n$ (assuming $\text{P}\neq\text{NP}$).\qed
\end{theorem}

\section{Embeddings Revisited}\label{EmbeddingsRe}

Recall that in Section \ref{Embeddings}, we constructed a Cayley graph $Y$ from: a graph $X$; a group $G$; and, a function $h:V(X)\rightarrow G$. Specifically we took $g_i=h(i)$, and $Y=X(G,\mathcal{C})$ where
\[
\mathcal{C}=\{g_ig_j^{-1}\,:\,\text{$i,j$ adjacent in $X$}\}.
\]
We saw that $h:X\rightarrow Y$ is a homomorphism, and that $h$ is an embedding if and only if $h$ is injective, and $g_ig_j^{-1}\notin\mathcal{C}$ for any $i,j$ non-adjacent in $X$.

This construction is used by Babai and S{\'o}s in \cite{BabaiSos} to answer the following question. Given a graph $X$, is there a group $G$ so that $X$ embeds in a Cayley graph for $G$? As we have seen the answer to this question is yes (e.g., $X$ embeds in the free Cayley graph $G(X)$ for any group $G$). In \cite{BabaiSos} the authors approach this question using Sidon sets of the second kind. As we noted in Section \ref{Embeddings} if such a set exists, then every graph on at most $v$ vertices embeds in a Cayley graph for $G$ (in particular, the Cayley graph $Y$ described above). They show that if $G$ is an arbitrary group, then $G$ contains a Sidon set of size at least $O(|G|^{1/3})$. This implies that if $X$ is a graph on $v$ vertices, then $X$ embeds in a Cayley graph for any group $G$ with $|G|=O(v^3)$.

We can ask the following related question: given a graph $X$ on $v$ vertices, what is the smallest Cayley graph in which $X$ embeds? For $p$ prime, Babai and S{\'o}s show that $\ints_p^{2n}$ contains a Sidon set of order $p^n$. This implies that $X$ on $v$ vertices embeds in a Cayley graph for $\ints_p^m$ where $p$ is some prime, and $p^m=O(v^2)$. Using free Cayley graphs, and quotienting over a code, we can derive a similar result.

For $p=2$, we have that $X$ embeds in $\ints_2(X)$. Recall from Proposition \ref{d5InducedX} that if $D\subseteq \ints_2^v$ is a code with distance $d\geq 5$, then $X$ embeds in $\ints_2(X)_D$. For $p=2$, we can use a binary BCH code (see \cite{MacWilliamsSloane} Chapter 3) with $d\geq 5$ to find a small cubelike graph containing an embedding of $X$.

\begin{proposition}\label{BCHExist}
For any $m\geq 3$ and $0\leq t<2^{m-1}$, there is a binary BCH code with length $n=2^m-1$, rank $k\geq n-mt$, and distance $d\geq 2t+1$.
\end{proposition}

\noindent From this fact, we have the immediate result.
\begin{theorem}\label{CubelikeQuadratic}
If $X$ is a graph on $v$ vertices, then there is a cubelike graph with $O(v^2)$ vertices that contains an induced copy of $X$.
\end{theorem}
\begin{proof}
Take $m$ to be the smallest integer with $v\leq 2^m-1$. Let $X'$ be the graph obtained by adding $2^m-1-v$ isolated vertices to $X$. Now $|X'|=v'< 2v$. Taking $t=2$ in Proposition \ref{BCHExist}, there is a BCH code $D$ in $\ints_2^{v'}$ with distance $d\geq 5$ and rank $k\geq v'-2\lfloor\log_2(v')\rfloor$. Now $\ints_2(X')_D$ contains an induced copy of $X'$, and hence an induced copy of $X$. Finally, $|\ints_2(X')_D|=2^{v'-k}\leq (v')^2$.\qed
\end{proof}

Theorem \ref{CubelikeQuadratic} shows that when $p=2$, using our approach we can construct Cayley graphs containing an induced copy of $X$ of the same order as those constructed in \cite{BabaiSos} (though with a worse coefficient). For $p\neq 2$ it may be possible to use $p$-ary BCH codes to obtain a similar result.

\section{Neighbourhood Structure}\label{Structure}

We finish with an observation on the structure of the neighbourhoods of the free Cayley graphs $G(X)$ for finite Abelian groups $G$. Since these graphs are vertex transitive, it will suffice to describe the neighbourhood of a particular vertex (for convenience we will use the vertex ${\bf 0}$, the vector whose components are all $0_G$).

For a graph $X$, define the \emph{triangle graph of $X$} to be the graph $T(X)$ with vertex set $E(X)$ where $e,f\in E(X)$ are adjacent if and only if they lie in a triangle in $X$. Note that $T(X)$ is a subgraph of $L(X)$.

\begin{lemma}\label{DoubleCover}
If $G$ is a finite Abelian group with exponent $m\geq 4$, then $G(X)[{\bf 0}]$ is a $2$-fold cover of $T(X)$.
\end{lemma}
\begin{proof}
Define $h:\mathcal{C}\rightarrow V(T(X))$ by
\[
h(g_i-g_j)=h(g_j-g_i)=\{i,j\}.
\]
From the definition of $\mathcal{C}$ we see that $h$ is clearly a surjection, and that $|h^{-1}(\{i,j\})|=2$ for all $\{i,j\}\in E(X)$. It remains to show that $h$ is a homomorphism, and a local isomorphism.

Recall from the proof of Lemma \ref{AbGroup} that $g_i-g_j$ is adjacent to $g_k-g_l$ in $G(X)[{\bf 0}]$ if and only if either $k=i$ or $j=l$ and $ijl$ or $ijk$ respectively is a triangle in $X$. Thus if $g_i-g_j$ is adjacent to $g_i-g_l$, then $\{i,j\}$ and $\{i,l\}$ are edges of $X$ that lie in a triangle. So $\{i,j\}$ is adjacent to $\{i,l\}$ in $T(X)$. (The case $g_i-g_j$ adjacent to $g_k-g_j$ is similar.) Therefore $h$ is a homomorphism.

Finally, consider the map induced by $h$ between the neighbours of $g_i-g_j$ and the neighbours of $\{i,j\}$. If $\{i,j\}$ is adjacent to $\{i,l\}$ in $T(X)$, then $g_i-g_j$ is adjacent to $g_i-g_l$ and $h(g_i-g_l)=\{i,l\}$, so $h$ induces a surjection. If $g_k-g_l$ and $g_s-g_t$ are both neighbours of $g_i-g_j$ in $G(X)$, then $h(g_k-g_l)=\{k,l\}$ and $h(g_s-g_t)=\{s,t\}$. If $\{k,l\}=\{s,t\}$, then either $k=s$ or $k=t$. If $k=s$, then
\[
g_k-g_l=g_s-g_t.
\]
If $k=t$, then
\[
g_k-g_l=-(g_s-g_t).
\]
However, in order to be a neighbour of $g_i-g_j$ we must have that either $k=i$ or $l=j$, and either $s=i$ or $t=j$. If $k=i$, then $t=i\neq j$ so we must have $s=i=t$ which is a contradiction. The other cases give similar contradictions. Thus $h$ induces an injection, and the induced map is a bijection.
\qed
\end{proof}

If the exponent of $G$ is two or three, Lemma \ref{DoubleCover} does not hold. However, for $p=2$ and $p=3$ we can still describe the structure of $\ints_p(X)[{\bf 0}]$, and prove that it is related to $T(X)$.

\begin{lemma}\label{DoubleCover2}
$\ints_2(X)[{\bf 0}]$ is isomorphic to $T(X)$.
\end{lemma}
\begin{proof}
Recall that for $p=2$, we have $|\mathcal{C}|=|E(G)|$. Define the function $h:\mathcal{C}\rightarrow V(T(X))$ by $h(g_i+g_j)=\{i,j\}$.

It is easy to see that $h$ is a bijection. Moreover, if $g_i+g_j$ is adjacent to $g_k+g_l$ in $\ints_2(X)[{\bf 0}]$, then without loss of generality, $j=k$ and $g_j+g_l\in\mathcal{C}$. Thus $ijl$ is a triangle in $X$ and $\{i,j\}$ and $\{j,l\}$ are adjacent in $T(X)$.\qed
\end{proof}

For $\ints_3$, recall from Section \ref{Zp} that the 3-sums of $\mathcal{C}$ are not distinct. However, the only problem is that $3(g_i-g_j)={\bf 0}$ for all $g_i-g_j\in\mathcal{C}$, or that $g_i-g_j$ is adjacent to $-(g_i-g_j)$ for all $g_i-g_j\in\mathcal{C}$. These edges give a perfect matching of $\ints_3(X)[{\bf 0}]$. Note that each of these edges lies in a fibre of the covering map $h$ defined in the proof of Lemma \ref{DoubleCover}. This leads us to the following Lemma, the proof of which we omit (the proof is exactly the same as the proof of Lemma \ref{DoubleCover}).

\begin{lemma}\label{DoubleCover3}
Let $M$ be the perfect matching of $\ints_3(X)[{\bf 0}]$ given by the edges $\{g_i-g_j,-(g_i-g_j)\}$ for all $\{i,j\}\in E(X)$. Then $\ints_3(X)[{\bf 0}]\setminus M$ is a $2$-fold cover of $T(X)$.\qed
\end{lemma}

\section{Open Problems}\label{OpenProblems}

There are a few obvious avenues of research suggested by our results; we address three of them here. First, we have proven that computing chromatic number cannot be done in polynomial time for a class of Cayley graphs. We were not able to prove a direct analogue of Codenotti et al.'s result for circulants.
\begin{problem}\label{NPHq}
If $G$ is a fixed finite group, is computing chromatic number NP-Hard for the class of Cayley graphs for the groups $G^n$?
\end{problem}

We have considered two computational problems that are NP-Hard for the class of all graphs. An immediate question is whether one's favourite NP-Hard computational problem for the class of all graphs remains NP-Hard when restricted to this class of Cayley graphs.

More restrictively, we might ask whether our construction using free Cayley graphs can be used to prove the NP-Hardness of any other problems. Motivated by Theorem \ref{ChromCor}, one might start with other flavours of colouring problems. The problem of computing the edge chromatic number of a graph $X$ is equivalent to computing $\chi(L(X))$. The connections between $G(X)$ and $L(X)$ given in Section \ref{Structure} suggest that edge chromatic number is worth consideration.
\begin{problem}\label{ECNq}
If $G$ is a fixed finite group, is computing edge chromatic number NP-Hard for the class of Cayley graphs for the groups $G^n$?
\end{problem}

Note that the Cayley graph $X=X(G,C)$ is $|C|$-regular. So by Vizing's Theorem $\chi'(X)$ is either $|C|$ or $|C|+1$. If $G=\ints_2^m$, then every element of the connection set corresponds to a perfect matching of $X$, and $\chi'(X)=|C|$. So Problem \ref{ECNq} is easily resolved for cubelike graphs.

Another avenue of research suggested by our results, is to further understand free Cayley graphs. Free Cayley graphs are a recent invention, and have been studied very little. Almost any question you could ask about these graphs is open. In our construction we showed that $\omega(G(X))$ is easily recoverable from $\omega(X)$, and that ``most of the time'' these parameters are equal. What about independence number, or chromatic number?
\begin{problem}
If $G$ is a finite group, can $\chi(G(X))$ be derived from $\chi(X)$?
\end{problem}
\noindent Answering this question would hopefully go some way to resolving open questions about the chromatic numbers of cubelike graphs.

\vspace{0.25cm}
\begin{center}{\bf Acknowledgements}\end{center}
\vspace{0.25cm}

The authors thank Noga Alon for his helpful comments; and an anonymous referee for some important corrections and improvements to the presentation.

\nocite{*}
\bibliographystyle{plain}
\bibliography{Complexity}

\end{document}